\documentclass[11pt]{article}
\usepackage{latexsym}
\usepackage{amsfonts}
\usepackage{mathrsfs,amsthm,amsmath}
\usepackage{color}
\usepackage{graphicx}
\newtheorem{Theorem}{Theorem}[section]
\newtheorem{Proposition}[Theorem]{Proposition}

\newtheorem{Remark}{Remark}[section]
\newtheorem{Claim}{Claim}[section]

\begin{document}
\title{Asymptotic results for the last zero crossing time of a Brownian
motion with non-null drift\thanks{The support of GNAMPA-INdAM and of 
MIUR Excellence Department Project awarded to the Department of 
Mathematics, University of Rome Tor Vergata (CUP E83C18000100006) is
acknowledged.}}
\author{Francesco Iafrate\thanks{Dipartimento di Scienze Statistiche, 
Sapienza Universit\`a di Roma, Piazzale Aldo Moro 5, I-00185 Rome,
Italy. e-mail: \texttt{francesco.iafrate@uniroma1.it}} \and 
Claudio Macci\thanks{Dipartimento di Matematica, Universit\`a
di Roma Tor Vergata, Via della Ricerca Scientifica, I-00133 Rome,
Italy. e-mail: \texttt{macci@mat.uniroma2.it}}}
\date{}
\maketitle
\begin{abstract}
\noindent We consider the last zero crossing time $T_{\mu,t}$ of 
a Brownian motion, with drift $\mu\neq 0$, in the time interval $[0,t]$.
We prove the large deviation principle of $\{T_{\mu\sqrt{r},t}:r>0\}$
as $r$ tends to infinity. Moreover, motivated by the results on moderate
deviations in the literature, we also prove a class of large deviation
principles for the same random variables with different scalings, which
are governed by the same rate function. Finally we compare some aspects
of the classical moderate deviation results, and the results in this 
paper.
\end{abstract}

\section{Introduction and preliminaries}
In this paper we consider a Brownian motion with drift, and we present some asymptotic 
results for the last zero crossing time (in a time interval $[0,t]$) as the drift tends
to infinity. In our proofs we handle some formulas presented in \cite{IafrateOrsingher},
and we refer to the theory of large deviations (see \cite{DemboZeitouni} as a reference
on this topic).

We recall some results in \cite{IafrateOrsingher}. We consider a Brownian motion 
(starting at the origin at time zero), with drift $\mu\neq 0$. Moreover, for some
$t>0$, we consider the last zero crossing time $T_{\mu,t}$ of this Brownian motion in 
the time interval $[0,t]$. Then (see Theorem 2.1 in \cite{IafrateOrsingher}) the 
distribution function of $T_{\mu,t}$ is defined by
\begin{equation}\label{eq:distribution-function}
P(T_{\mu,t}\leq a)=1-\frac{2}{\pi}\int_0^{\sqrt{\frac{t-a}{a}}}\frac{e^{-\frac{\mu^2}{2}a(1+y^2)}}{1+y^2}dy\ (\mbox{for}\ a\in[0,t]),
\end{equation}
and its density is
\begin{equation}\label{eq:density}
f_{\mu,t}(a):=\left(\frac{e^{-\frac{\mu^2}{2}t}}{\pi\sqrt{a(t-a)}}
+\frac{\mu^2}{2\pi}\int_a^t\frac{e^{-\frac{\mu^2}{2}y}}{\sqrt{a(y-a)}}dy\right)1_{[0,t]}(a).
\end{equation}
In this paper we also consider the probability that this Brownian motion crosses zero in 
the time interval $[a,b]$ for some $b>a>0$ (see the proof of Theorem 3.1 in 
\cite{IafrateOrsingher}):
\begin{equation}\label{eq:crossing-probability}
\Psi_{[a,b]}(\mu)=\frac{1}{\pi}\int_0^{b-a}\frac{e^{-\frac{\mu^2}{2}(a+s)}}{a+s}\sqrt{\frac{a}{s}}ds.
\end{equation}

We also recall the concept of large deviation principle (LDP for short) for a family of
random variables $\{W(r):r>0\}$ defined on the same probability space $(\Omega,\mathcal{F},P)$.
In view of what follows we assume that these random variables are real valued. Then $\{W(r):r>0\}$
satisfies the large deviation principle (LDP from now on) with speed $v_r$ and rate function $I$ 
if the following conditions hold: $\lim_{r\to\infty}v_r=\infty$, the function $I:\mathbb{R}\to [0,\infty]$ is 
lower semi-continuous;
\begin{equation}\label{eq:UB}
\limsup_{r\to\infty}\frac{1}{v_r}\log P(W(r)\in C)\leq-\inf_{w\in C}I(w)\ \mbox{for all closed sets}\ C;
\end{equation}
\begin{equation}\label{eq:LB}
\liminf_{r\to\infty}\frac{1}{v_r}\log P(W(r)\in G)\geq-\inf_{w\in G}I(w)\ \mbox{for all open sets}\ G.
\end{equation}
Moreover, a rate function $I$ is said to be good if all its level sets
$\{\{w\in\mathbb{R}:I(w)\leq\eta\}:\eta\geq 0\}$ are compact.

Now we present a list of the results proved in this paper.
\begin{itemize}
\item The LDP for $W(r)=T_{\mu\sqrt{r},t}$ with $v_r=r$ and a rate function $J$ 
(Proposition \ref{prop:LD}).
\item For every choice of positive numbers 
$\{\gamma_r:r>0\}$ such that
\begin{equation}\label{eq:conditions-MD}
\lim_{r\to\infty}\gamma_r=0\ \mbox{and}\ \lim_{r\to\infty}r\gamma_r=\infty,
\end{equation}
the LDP for $W(r)=r\gamma_rT_{\mu\sqrt{r},t}$ with $v_r=1/\gamma_r$ and the same
rate function $\tilde{J}$ (Proposition \ref{prop:MD}). It will be explained that
this class of LDPs is inspired by the results in the literature on moderate 
deviations.
\item We study the asymptotic behavior, as $r\to\infty$, for the probability in 
\eqref{eq:crossing-probability} with $\mu\sqrt{r}$ in place of $\mu$ (Proposition
\ref{prop:crossing-probability}). We also highlight some similar aspects with
asymptotic estimates in insurance literature (see Remark \ref{rem:analogies-with-ruin-probability}).
\end{itemize}

Throughout the paper, we consider the notation 
$$\varphi(x;\mu,\sigma^2)=\frac{e^{-\frac{(x-\mu)^2}{2\sigma^2}}}{\sqrt{2\pi\sigma^2}}$$
for the well-known density of a Normal distribution with mean $\mu$ and variance 
$\sigma^2$. 

We conclude with the outline of the paper. In Section \ref{sec:results}, we prove the results.
Finally, in Section \ref{sec:MD-comments}, we compare some aspects of the classical moderate
deviation results, and the results in this paper; in particular we discuss some common 
features and differences.

\section{Results}\label{sec:results}
We start with the first result, i.e. a LDP with $v_r=r$. We remark that, since we deal with random 
variables $\{W(r):r>0\}$ which take values on a compact interval $K$ of the real line (i.e. 
$K=[0,t]$), we can refer to a useful consequence of Theorem 4.1.11 in \cite{DemboZeitouni}. In fact,
we can say that we have the LDP if the two following conditions hold for all 
$w\in\mathbb{R}$:
\begin{equation}\label{eq:forLB}
\lim_{\varepsilon\to 0}\liminf_{r\to\infty}\frac{1}{r}\log P(W(r)\in (w-\varepsilon,w+\varepsilon))\geq-I(w);
\end{equation}
\begin{equation}\label{eq:forUB}
\lim_{\varepsilon\to 0}\limsup_{r\to\infty}\frac{1}{r}\log P(W(r)\in (w-\varepsilon,w+\varepsilon))\leq-I(w).
\end{equation}
Actually, when $w\notin K$, these two bounds can be easily checked with $I(w)=\infty$
(\eqref{eq:forLB} trivially holds; moreover, if we take $\varepsilon>0$ small enough to have 
$(w-\varepsilon,w+\varepsilon)\cap K=\emptyset$, we have $P(W(r)\in(w-\varepsilon,w+\varepsilon))=0$
for all $r>0$, which yields \eqref{eq:forUB}); so we can consider only the case $w\in K$.

\begin{Proposition}\label{prop:LD}
The family of random variables $\{T_{\mu\sqrt{r},t}:r>0\}$ satisfies the 
LDP with speed $v_r=r$ and good rate function $J$ defined by
$$J(b):=\left\{\begin{array}{ll}
\frac{\mu^2}{2}b&\ if\ b\in[0,t]\\
\infty&\ otherwise.
\end{array}\right.$$
\end{Proposition}
\begin{proof}
We prove the LDP by referring to the consequence of Theorem 4.1.11 in
\cite{DemboZeitouni} recalled above; so, for all $b\in [0,t]$ (here $[0,t]$
plays the role of compact interval $K$), we have to check \eqref{eq:forLB} and 
\eqref{eq:forUB} with $T_{\mu\sqrt{r},t}$ and $J(b)$ in place of $W(r)$ and $I(w)$,
respectively. We have three different cases.

\paragraph{Case $b\in(0,t)$.} Without loss of generality we can take $\varepsilon>0$ 
small enough to have $(b-\varepsilon,b+\varepsilon)\subset(0,t)$. Then, there exists
$\tilde{a}_r=\tilde{a}_r(\varepsilon,b)\in(b-\varepsilon,b+\varepsilon)$
such that
\begin{equation}\label{eq:mean-value-theorem}
P(T_{\mu\sqrt{r},t}\in(b-\varepsilon,b+\varepsilon))=\int_{b-\varepsilon}^{b+\varepsilon}f_{\mu\sqrt{r},t}(a)da
=f_{\mu\sqrt{r},t}(\tilde{a}_r)2\varepsilon.
\end{equation}
For the proof of \eqref{eq:forLB} we take into account \eqref{eq:density}, and after 
some manipulations, we get
\begin{multline*}
f_{\mu\sqrt{r},t}(\tilde{a}_r)\geq\frac{1}{\pi\sqrt{(b+\varepsilon)(t-(b-\varepsilon))}}\int_{\tilde{a}_r}^t\frac{\mu^2r}{2}e^{-\frac{\mu^2r}{2}y}dy\\
=\frac{e^{-\frac{\mu^2r}{2}\tilde{a}_r}-e^{-\frac{\mu^2r}{2}t}}{\pi\sqrt{(b+\varepsilon)(t-(b-\varepsilon))}}
=\frac{e^{-\frac{\mu^2r}{2}\tilde{a}_r}(1-e^{-\frac{\mu^2r}{2}(t-\tilde{a}_r)})}{\pi\sqrt{(b+\varepsilon)(t-(b-\varepsilon))}}
\geq\frac{e^{-\frac{\mu^2r}{2}(b+\varepsilon)}(1-e^{-\frac{\mu^2r}{2}(t-(b+\varepsilon))})}{\pi\sqrt{(b+\varepsilon)(t-(b-\varepsilon))}};
\end{multline*}
thus, by \eqref{eq:mean-value-theorem}, we have
$$\liminf_{r\to\infty}\frac{1}{r}\log P(T_{\mu\sqrt{r},t}\in (b-\varepsilon,b+\varepsilon))\geq-\frac{\mu^2}{2}(b+\varepsilon),$$
and therefore
$$\lim_{\varepsilon\to 0}\liminf_{r\to\infty}\frac{1}{r}\log P(T_{\mu\sqrt{r},t}\in (b-\varepsilon,b+\varepsilon))\geq-\frac{\mu^2}{2}b=-J(b).$$
For the proof of \eqref{eq:forUB} we take into account \eqref{eq:density} and, after 
some manipulations, we get
\begin{multline*}
f_{\mu\sqrt{r},t}(\tilde{a}_r)\leq
\frac{e^{-\frac{\mu^2r}{2}t}}{\pi\sqrt{(b-\varepsilon)(t-(b+\varepsilon))}}
+\frac{\mu^2r}{2\pi}\int_{\tilde{a}_r}^t\frac{e^{-\frac{\mu^2r}{2}y}}{\sqrt{\tilde{a}_r(y-\tilde{a}_r)}}dy\\
\leq\frac{e^{-\frac{\mu^2r}{2}t}}{\pi\sqrt{(b-\varepsilon)(t-(b+\varepsilon))}}
+\frac{\mu^2r}{2\pi}\frac{e^{-\frac{\mu^2r}{2}\tilde{a}_r}}{\sqrt{b-\varepsilon}}\int_{\tilde{a}_r}^t\frac{1}{\sqrt{y-\tilde{a}_r}}dy\\
\leq\frac{e^{-\frac{\mu^2r}{2}t}}{\pi\sqrt{(b-\varepsilon)(t-(b+\varepsilon))}}
+\frac{\mu^2r}{2\pi}\frac{e^{-\frac{\mu^2r}{2}(b-\varepsilon)}}{\sqrt{b-\varepsilon}}2\sqrt{t-(b-\varepsilon)};
\end{multline*}
thus, by \eqref{eq:mean-value-theorem} and by Lemma 1.2.15 in \cite{DemboZeitouni}, we have
$$\limsup_{r\to\infty}\frac{1}{r}\log P(T_{\mu\sqrt{r},t}\in (b-\varepsilon,b+\varepsilon))
\leq\max\left\{-\frac{\mu^2}{2}t,-\frac{\mu^2}{2}(b-\varepsilon)\right\}=-\frac{\mu^2}{2}(b-\varepsilon),$$
and therefore
$$\lim_{\varepsilon\to 0}\limsup_{r\to\infty}\frac{1}{r}\log P(T_{\mu\sqrt{r},t}\in (b-\varepsilon,b+\varepsilon))\leq-\frac{\mu^2}{2}b=-J(b).$$

\paragraph{Case $b=0$.} We can take $\varepsilon>0$ small enough to have 
$(0,\varepsilon)\subset(0,t)$. Then, there exists
$\tilde{a}_r=\tilde{a}_r(\varepsilon)\in(0,\varepsilon)$
such that
\begin{equation}\label{eq:mean-value-theorem-b=0}
P(T_{\mu\sqrt{r},t}\in(0,\varepsilon))=\int_0^\varepsilon f_{\mu\sqrt{r},t}(a)da
=f_{\mu\sqrt{r},t}(\tilde{a}_r)\varepsilon.
\end{equation}
For the proof of \eqref{eq:forLB} we take into account \eqref{eq:density}, and after 
some manipulations, we get
\begin{multline*}
f_{\mu\sqrt{r},t}(\tilde{a}_r)\geq\frac{1}{\pi\sqrt{\varepsilon t}}\int_{\tilde{a}_r}^t\frac{\mu^2r}{2}e^{-\frac{\mu^2r}{2}y}dy\\
=\frac{e^{-\frac{\mu^2r}{2}\tilde{a}_r}-e^{-\frac{\mu^2r}{2}t}}{\pi\sqrt{\varepsilon t}}
=\frac{e^{-\frac{\mu^2r}{2}\tilde{a}_r}(1-e^{-\frac{\mu^2r}{2}(t-\tilde{a}_r)})}{\pi\sqrt{\varepsilon t}}
\geq\frac{e^{-\frac{\mu^2r}{2}\varepsilon}(1-e^{-\frac{\mu^2r}{2}(t-\varepsilon)})}{\pi\sqrt{\varepsilon t}};
\end{multline*}
thus, by \eqref{eq:mean-value-theorem-b=0}, we have
$$\liminf_{r\to\infty}\frac{1}{r}\log P(T_{\mu\sqrt{r},t}\in (0,\varepsilon))\geq-\frac{\mu^2}{2}\varepsilon,$$
and therefore, since $P(T_{\mu\sqrt{r},t}\in (0,\varepsilon))=P(T_{\mu\sqrt{r},t}\in (-\varepsilon,\varepsilon))$,
$$\lim_{\varepsilon\to 0}\liminf_{r\to\infty}\frac{1}{r}\log P(T_{\mu\sqrt{r},t}\in (-\varepsilon,\varepsilon))\geq 0=-J(0).$$
The proof of \eqref{eq:forUB} is immediate; in fact we have
$$\lim_{\varepsilon\to 0}\limsup_{r\to\infty}\frac{1}{r}\log P(T_{\mu\sqrt{r},t}\in (-\varepsilon,\varepsilon))\leq 0=-J(0),$$
noting that $\frac{1}{r}\log P(T_{\mu\sqrt{r},t}\in (-\varepsilon,\varepsilon))\leq 0$
for all $r>0$.

\paragraph{Case $b=t$.} We can take $\varepsilon>0$ small enough to have 
$(t-\varepsilon,t)\subset(0,t)$. Then, by \eqref{eq:density}, we have
\begin{equation}\label{eq:starting-point-b=t}
P(T_{\mu\sqrt{r},t}\in(t-\varepsilon,t))=\int_{t-\varepsilon}^tf_{\mu\sqrt{r},t}(a)da
=\int_{t-\varepsilon}^t\left(\frac{e^{-\frac{\mu^2r}{2}t}}{\pi\sqrt{a(t-a)}}
+\frac{\mu^2r}{2\pi}\int_a^t\frac{e^{-\frac{\mu^2r}{2}y}}{\sqrt{a(y-a)}}dy\right)da.
\end{equation}
For the proof of \eqref{eq:forLB} we take into account \eqref{eq:starting-point-b=t}, and after 
some manipulations, we get
$$P(T_{\mu\sqrt{r},t}\in(t-\varepsilon,t))\geq\int_{t-\varepsilon}^t\frac{e^{-\frac{\mu^2r}{2}t}}{\pi\sqrt{a(t-a)}}da
\geq\frac{e^{-\frac{\mu^2r}{2}t}}{\pi\sqrt{t}}\int_{t-\varepsilon}^t\frac{1}{\sqrt{t-a}}da
=\frac{e^{-\frac{\mu^2r}{2}t}}{\pi\sqrt{t}}2\sqrt{\varepsilon};$$
thus
$$\liminf_{r\to\infty}\frac{1}{r}\log P(T_{\mu\sqrt{r},t}\in (t-\varepsilon,t))\geq-\frac{\mu^2}{2}t,$$
and therefore, since $P(T_{\mu\sqrt{r},t}\in (t-\varepsilon,t))=P(T_{\mu\sqrt{r},t}\in (t-\varepsilon,t+\varepsilon))$,
$$\lim_{\varepsilon\to 0}\liminf_{r\to\infty}\frac{1}{r}\log P(T_{\mu\sqrt{r},t}\in (t-\varepsilon,t+\varepsilon))\geq-\frac{\mu^2}{2}t=-J(t).$$
For the proof of \eqref{eq:forUB} we take into account \eqref{eq:starting-point-b=t} and, after 
some manipulations, we get
\begin{multline*}
P(T_{\mu\sqrt{r},t}\in(t-\varepsilon,t))=\frac{e^{-\frac{\mu^2r}{2}t}}{\pi}\int_{t-\varepsilon}^t\frac{1}{\sqrt{a(t-a)}}da
+\frac{\mu^2r}{2\pi}\int_{t-\varepsilon}^tda\int_a^tdy\frac{e^{-\frac{\mu^2r}{2}y}}{\sqrt{a(y-a)}}\\
\leq\frac{e^{-\frac{\mu^2r}{2}t}}{\pi\sqrt{t-\varepsilon}}\int_{t-\varepsilon}^t\frac{1}{\sqrt{t-a}}da
+\frac{\mu^2r}{2\pi}\int_{t-\varepsilon}^tda\frac{e^{-\frac{\mu^2r}{2}a}}{\sqrt{a}}\int_a^tdy\frac{1}{\sqrt{y-a}}\\
\leq\frac{e^{-\frac{\mu^2r}{2}t}}{\pi\sqrt{t-\varepsilon}}2\sqrt{\varepsilon}
+\frac{\mu^2r}{2\pi}\int_{t-\varepsilon}^tda\frac{e^{-\frac{\mu^2r}{2}a}}{\sqrt{a}}2\sqrt{t-a}
\leq\frac{e^{-\frac{\mu^2r}{2}t}}{\pi\sqrt{t-\varepsilon}}2\sqrt{\varepsilon}
+\frac{\mu^2r}{2\pi}\frac{e^{-\frac{\mu^2r}{2}(t-\varepsilon)}}{\sqrt{t-\varepsilon}}2\sqrt{\varepsilon};
\end{multline*}
thus, by Lemma 1.2.15 in \cite{DemboZeitouni}, we have
$$\limsup_{r\to\infty}\frac{1}{r}\log P(T_{\mu\sqrt{r},t}\in (t-\varepsilon,t))
\leq\max\left\{-\frac{\mu^2}{2}t,-\frac{\mu^2}{2}(t-\varepsilon)\right\}=-\frac{\mu^2}{2}(t-\varepsilon),$$
and therefore, since $P(T_{\mu\sqrt{r},t}\in (t-\varepsilon,t))=P(T_{\mu\sqrt{r},t}\in (t-\varepsilon,t+\varepsilon))$,
$$\lim_{\varepsilon\to 0}\limsup_{r\to\infty}\frac{1}{r}\log P(T_{\mu\sqrt{r},t}\in (t-\varepsilon,t+\varepsilon))\leq-\frac{\mu^2}{2}t=-J(t).$$
\end{proof}

Now we present the class of LDPs inspired by moderate deviations.

\begin{Proposition}\label{prop:MD}
For every choice of positive numbers $\{\gamma_r:r>0\}$ such that
\eqref{eq:conditions-MD} holds, the family of random variables 
$\{r\gamma_rT_{\mu\sqrt{r},t}:r>0\}$ satisfies the 
LDP with speed $v_r=1/\gamma_r$ and good rate function $\tilde{J}$ defined by
$$\tilde{J}(b):=\left\{\begin{array}{ll}
\frac{\mu^2}{2}b&\ if\ b\in[0,\infty)\\
\infty&\ otherwise.
\end{array}\right.$$
\end{Proposition}
\begin{proof}
We can restrict the attention on the case $b\geq 0$ because we deal with
nonnegative random variables (and $[0,\infty)$ is a closed set). We assume
for the moment that we have
\begin{equation}\label{eq:half-lines}
\lim_{r\to\infty}\gamma_r\log P(r\gamma_rT_{\mu\sqrt{r},t}\geq z)=-\frac{\mu^2}{2}z\ (\mbox{for all}\ z\geq 0).
\end{equation}
We have to check the upper bound \eqref{eq:UB} and the lower bound \eqref{eq:LB}
with $v_r=1/\gamma_r$ and $W(r)=r\gamma_rT_{\mu\sqrt{r},t}$.
\begin{itemize}
\item The upper bound \eqref{eq:UB} trivially holds if $C\cap [0,\infty)$ is empty.
On the contrary, there exists $z_C:=\min(C\cap [0,\infty))$ (the existence of $z_C$ 
is guaranteed because $C\cap [0,\infty)$ is a non-empty closed set), and we have
$$P(r\gamma_rT_{\mu\sqrt{r},t}\in C)\leq P(r\gamma_rT_{\mu\sqrt{r},t}\geq z_C);$$
thus, by \eqref{eq:half-lines} and by the monotonicity of $\tilde{J}(b)$ on $[0,\infty)$, we get
$$\limsup_{r\to\infty}\frac{1}{1/\gamma_r}\log P(r\gamma_rT_{\mu\sqrt{r},t}\in C)\leq
\limsup_{r\to\infty}\gamma_r\log P(r\gamma_rT_{\mu\sqrt{r},t}\geq z_C)=-\frac{\mu^2}{2}z_C=-\inf_{b\in C}\tilde{J}(b).$$
\item It is known (see e.g. \cite{DemboZeitouni}, condition (b) with equation (1.2.8)) that the 
lower bound \eqref{eq:LB} holds if and only if, for all $b\geq 0$ and for all open sets $G$ such
that $b\in G$, we have
\begin{equation}\label{eq:LB-equivalent}
\liminf_{r\to\infty}\frac{1}{1/\gamma_r}\log P(r\gamma_rT_{\mu\sqrt{r},t}\in G)\geq-\tilde{J}(b).
\end{equation}
In order to get this condition we remark that there exists $\varepsilon>0$ such that 
$(b-\varepsilon,b+\varepsilon)\subset G$, and we have
$$P(r\gamma_rT_{\mu\sqrt{r},t}\in G)\geq 
P(r\gamma_rT_{\mu\sqrt{r},t}\geq b-\varepsilon)-P(r\gamma_rT_{\mu\sqrt{r},t}\geq b+\varepsilon).$$
Then, by a suitable application of Lemma 19 in \cite{GaneshTorrisi} (and by taking into account 
\eqref{eq:half-lines}), we get
\begin{multline*}
\liminf_{r\to\infty}\frac{1}{1/\gamma_r}\log P(r\gamma_rT_{\mu\sqrt{r},t}\in G)\\
\geq\liminf_{r\to\infty}\gamma_r\log (P(r\gamma_rT_{\mu\sqrt{r},t}\geq b-\varepsilon)-P(r\gamma_rT_{\mu\sqrt{r},t}\geq b+\varepsilon))
\geq-\frac{\mu^2}{2}(b-\varepsilon).
\end{multline*}
Thus we get \eqref{eq:LB-equivalent} by letting $\varepsilon$ go to zero.
\end{itemize}
In conclusion, we complete the proof showing that \eqref{eq:half-lines} holds.
The case $z=0$ is trivial, and therefore we take $z>0$. We take $r$ large enough
such that $\frac{z}{r\gamma_r}\in[0,t]$ (we recall that $\lim_{r\to\infty}r\gamma_r=\infty$)
and, by \eqref{eq:distribution-function}, we have
$$P(r\gamma_rT_{\mu\sqrt{r},t}\geq z)=P\left(T_{\mu\sqrt{r},t}\geq\frac{z}{r\gamma_r}\right)=
\frac{2}{\pi}\int_0^{\sqrt{\frac{t-z/(r\gamma_r)}{z/(r\gamma_r)}}}\frac{e^{-\frac{\mu^2z}{2\gamma_r}(1+y^2)}}{1+y^2}dy.$$
Then
$$P(r\gamma_rT_{\mu\sqrt{r},t}\geq z)=
\frac{e^{-\frac{\mu^2z}{2\gamma_r}}}{\pi}\int_{-\sqrt{\frac{t-z/(r\gamma_r)}{z/(r\gamma_r)}}}^{\sqrt{\frac{t-z/(r\gamma_r)}{z/(r\gamma_r)}}}
\frac{e^{-\frac{\mu^2z}{2\gamma_r}y^2}}{1+y^2}dy
=\sqrt{\frac{2\gamma_r}{\pi\mu^2z}}e^{-\frac{\mu^2z}{2\gamma_r}}
\int_{-\sqrt{\frac{t-z/(r\gamma_r)}{z/(r\gamma_r)}}}^{\sqrt{\frac{t-z/(r\gamma_r)}{z/(r\gamma_r)}}}
\frac{\varphi(y;0,\gamma_r/(\mu^2z))}{1+y^2}dy,$$
and
$$\gamma_r\log P(r\gamma_rT_{\mu\sqrt{r},t}\geq z)
=\gamma_r\log\sqrt{\frac{2\gamma_r}{\pi\mu^2z}}-\frac{\mu^2}{2}z+\gamma_r\log
\int_{-\sqrt{\frac{t-z/(r\gamma_r)}{z/(r\gamma_r)}}}^{\sqrt{\frac{t-z/(r\gamma_r)}{z/(r\gamma_r)}}}
\frac{\varphi(y;0,\gamma_r/(\mu^2z))}{1+y^2}dy;$$
thus we complete the proof of \eqref{eq:half-lines} showing that
\begin{equation}\label{eq:final-step}
\lim_{r\to\infty}\int_{-\sqrt{\frac{t-z/(r\gamma_r)}{z/(r\gamma_r)}}}^{\sqrt{\frac{t-z/(r\gamma_r)}{z/(r\gamma_r)}}}
\frac{\varphi(y;0,\gamma_r/(\mu^2z))}{1+y^2}dy=1.
\end{equation}
In order to do that we remark that, by the triangular inequality and after some easy manipulations,
we get
$$\left|\int_{-\sqrt{\frac{t-z/(r\gamma_r)}{z/(r\gamma_r)}}}^{\sqrt{\frac{t-z/(r\gamma_r)}{z/(r\gamma_r)}}}
\frac{\varphi(y;0,\gamma_r/(\mu^2z))}{1+y^2}dy-1\right|\leq A_1(r)+A_2(r)$$
where
$$A_1(r):=\left|\int_{-\infty}^\infty\frac{\varphi(y;0,\gamma_r/(\mu^2z))}{1+y^2}dy-1\right|$$
and
$$A_2(r):=2\int_{\sqrt{\frac{t-z/(r\gamma_r)}{z/(r\gamma_r)}}}^\infty\frac{\varphi(y;0,\gamma_r/(\mu^2z))}{1+y^2}dy.$$
Then \eqref{eq:final-step} holds (and this completes the proof) noting that:
\begin{itemize}
\item $\lim_{r\to\infty}A_1(r)=0$ by the weak convergence of the centered Normal distribution with 
variance $\frac{\gamma_r}{\mu^2z}$ to zero (as $r\to\infty$);
\item $\lim_{r\to\infty}A_2(r)=0$ by taking into account that
$$0\leq A_2(r)\leq\frac{2\frac{\gamma_r}{\mu^2z}}{\sqrt{2\pi}\sqrt{\frac{t-z/(r\gamma_r)}{z/(r\gamma_r)}}}
\exp\left(-\frac{t-z/(r\gamma_r)}{2z/(r\gamma_r)}\frac{\mu^2z}{\gamma_r}\right)$$
by a well-known estimate of the tail of Gaussian distribution.
\end{itemize}
\end{proof}

We conclude with the minor result for the crossing probability in 
\eqref{eq:crossing-probability}.

\begin{Proposition}\label{prop:crossing-probability}
We consider $b>a>0$. Then
$$\lim_{r\to\infty}\frac{1}{r}\log\Psi_{[a,b]}(\mu\sqrt{r})=-\frac{\mu^2}{2}a\ \mbox{and}\ \lim_{r\to\infty}e^{\frac{\mu^2ra}{2}}\sqrt{r}\Psi_{[a,b]}(\mu\sqrt{r})=\sqrt{\frac{2}{\pi\mu^2a}}.$$
\end{Proposition}
\begin{proof}
We have
$$\Psi_{[a,b]}(\mu\sqrt{r})=\frac{1}{\pi}\int_0^{b-a}\frac{e^{-\frac{\mu^2r}{2}(a+s)}}{a+s}\sqrt{\frac{a}{s}}ds
=\frac{2}{\pi}\int_0^{\sqrt{\frac{b-a}{a}}}\frac{e^{-\frac{\mu^2ra}{2}(1+y^2)}}{1+y^2}dy$$
(in the second equality we take into account the change of variable $s=ay^2$). Thus
$$\Psi_{[a,b]}(\mu\sqrt{r})=\frac{2}{\pi}e^{-\frac{\mu^2ra}{2}}\int_0^{\sqrt{\frac{b-a}{a}}}\frac{e^{-\frac{\mu^2ra}{2}y^2}}{1+y^2}dy
=\sqrt{\frac{2}{\pi\mu^2ra}}e^{-\frac{\mu^2ra}{2}}\int_{-\sqrt{\frac{b-a}{a}}}^{\sqrt{\frac{b-a}{a}}}\frac{\varphi(y;0,1/(\mu^2ra))}{1+y^2}dy.$$
Since
$$\lim_{r\to\infty}\int_{-\sqrt{\frac{b-a}{a}}}^{\sqrt{\frac{b-a}{a}}}\frac{\varphi(y;0,1/(\mu^2ra))}{1+y^2}dy=1,$$
then, by the weak convergence of the centered Normal distribution with variance $\frac{1}{\mu^2ra}$
to zero (as $r\to\infty$), we can get the desired limits with some easy computations.
\end{proof}

\begin{Remark}\label{rem:analogies-with-ruin-probability}
There are some similarities between the limits in Proposition 
\ref{prop:crossing-probability} and some asymptotic estimates of
level crossing probabilities in the literature. For instance, if
we denote the level crossing probability by $\psi(r)$ (here $r>0$
is the level), under suitable conditions (see e.g. 
\cite{DuffyLewisSullivan}) we have
\begin{equation}\label{eq:DLS}
\lim_{r\to\infty}\frac{1}{h(r)}\log\psi(r)=-w
\end{equation}
for some $w>0$ and some scaling function $h(\cdot)$. For instance
here we recall the case of the Cram\'er-Lundberg model in insurance
(see e.g. \cite{AlbrecherAsmussen}), where $\psi(r)$ is interpreted
as the ruin probability and $r$ as the initial capital. Then, under
suitable hypotheses, we have the two following statements:
\begin{itemize}
\item for some $w>0$ and some $c_1,c_2\in(0,1]$ with $c_1\leq c_2$ 
(see e.g. Theorem 6.3 in \cite{AlbrecherAsmussen}, Chapter IV), we have
$$c_1e^{-wr}\leq\psi(r)\leq c_2e^{-wr},$$
which yields \eqref{eq:DLS} with $h(r)=r$;
\item for some $c>0$ (see e.g. (4.3) in \cite{AlbrecherAsmussen}, 
Chapter I; see also Theorem 1.2.2(b) in \cite{EmbrechtsKluppelbergMikosch})
we have
\begin{equation}\label{eq:CL-approximation}
\lim_{r\to\infty}e^{wr}\psi(r)=c.
\end{equation}
\end{itemize}
So the limits in Proposition \ref{prop:crossing-probability} for 
$\Psi_{[a,b]}(\mu\sqrt{r})$ have some relationship with the limits here
for $\psi(r)$. However the scaling factor $e^{\frac{\mu^2ra}{2}}\sqrt{r}$
(see the second limit in Proposition \ref{prop:crossing-probability}) is 
different from $e^{wr}$ in \eqref{eq:CL-approximation}.
\end{Remark}

\section{Comparison with moderate deviation results in the literature}\label{sec:MD-comments}
The term moderate deviations is used in the literature for a class of LDPs for suitable
centered random variables, and governed by the same quadratic rate function (here we 
restrict the attention on real valued random variables for simplicity but, actually, 
a similar concept can be given for vector valued random variables). Proposition 
\ref{prop:MD} also provides a class of LDPs: the random variables are not centered,
but they converge to zero because the rate function $J$ in Proposition \ref{prop:LD} 
uniquely vanishes at zero. So in this section we want to discuss analogies and differences 
between the moderate deviations results in the literature, usually related to the use of
G\"{a}rtner Ellis Theorem (see e.g. Theorem 2.3.6 in \cite{DemboZeitouni}), and
the results in this paper.

We start with Claim \ref{claim:LDMD}, which provides the usual framework for both
large and moderate deviations for a family of random variables $\{W(r):r>0\}$.
There is an initial LDP, and a class of LDPs which concerns moderate deviations.
One can immediately see the analogies with the statements of Propositions 
\ref{prop:LD} and \ref{prop:MD} in this paper; in particular 
\eqref{eq:conditions-MD-gen} below plays the role of \eqref{eq:conditions-MD} in 
Proposition \ref{prop:MD}.

\begin{Claim}\label{claim:LDMD}
We assume that, for all $\theta\in\mathbb{R}$,
$$\Lambda(\theta):=\lim_{r\to\infty}\frac{1}{v_r}\log\mathbb{E}[e^{v_r\theta W(r)}]$$
exists as an extended real number. Then, under suitable hypotheses (see e.g. part
(c) of Theorem 2.3.6 in \cite{DemboZeitouni}) the LDP holds with speed $v_r\to\infty$ 
and good rate function $\Lambda^*$ defined by
$$\Lambda^*(w):=\sup_{\theta\in\mathbb{R}}\left\{\theta w-\Lambda(\theta)\right\}.$$

Furthermore, we set $\tilde{\Lambda}(\theta):=\frac{\theta^2}{2}\Lambda^{\prime\prime}(0)$,
where $\Lambda$ is the function above, and $\Lambda^{\prime\prime}$ is its second derivative
(note that $\Lambda^{\prime\prime}(0)\geq 0$ because $\Lambda$ is a convex function). 
Then, for every choice of positive numbers $\{\gamma_r:r>0\}$ such that
\begin{equation}\label{eq:conditions-MD-gen}
\lim_{r\to\infty}\gamma_r=0\ \mbox{and}\ \lim_{r\to\infty}v_r\gamma_r=\infty,
\end{equation}
we can prove that
$$\lim_{r\to\infty}\frac{1}{1/\gamma_r}\log\mathbb{E}[e^{(1/\gamma_r)\theta\sqrt{v_r\gamma_r}(W(r)-\mathbb{E}[W(r)])}]=
\tilde{\Lambda}(\theta)$$
for all $\theta\in\mathbb{R}$; thus $\left\{\sqrt{v_r\gamma_r}(W(r)-\mathbb{E}[W(r)]):r\geq 1\right\}$
satisfies the LDP with speed $1/\gamma_r$ and good rate function $\tilde{\Lambda}^*$ defined by
$$\tilde{\Lambda}^*(w):=\sup_{\theta\in\mathbb{R}}\left\{\theta w-\tilde{\Lambda}(\theta)\right\}=
\left\{\begin{array}{ll}
\frac{w^2}{2\Lambda^{\prime\prime}(0)}&\ \mbox{if}\ \Lambda^{\prime\prime}(0)>0\\
\left\{\begin{array}{ll}
0&\ \mbox{if}\ w=0\\
\infty&\ \mbox{if}\ w\neq 0
\end{array}\right.
&\ \mbox{if}\ \Lambda^{\prime\prime}(0)=0.
\end{array}\right.$$
\end{Claim}

\begin{Remark}\label{rem:Taylor-formula-order-2}
When $\Lambda^{\prime\prime}(0)>0$, the Taylor formula of order 2 of $\Lambda^*$, and 
initial point $\Lambda^\prime(0)$, is $\tilde{\Lambda}^*(w-\Lambda^\prime(0))$. A 
similar relationship concerns the Mac Laurin formula of order 2 of $\Lambda$, that 
is $\theta\Lambda^\prime(0)+\tilde{\Lambda}(\theta)$.
\end{Remark}

In the framework of Claim \ref{claim:LDMD} we also have the following typical features.

\begin{Claim}\label{claim:mean-variance-fill-the-gap}
Firstly we have
\begin{equation}\label{eq:mean-variance-formula}
\lim_{r\to\infty}\mathbb{E}[W(r)]=\Lambda^\prime(0)\ \mbox{and}\ \lim_{r\to\infty}v_r\mathrm{Var}[W(r)]=\Lambda^{\prime\prime}(0).
\end{equation}
Moreover, we can say that moderate deviations fill the gap between
two different regimes (as $r\to\infty$):
\begin{itemize}
\item the convergence of $W(r)-\mathbb{E}[W(r)]$ to zero (case 
$\gamma_r=1/v_r$; note that only the first condition in 
\eqref{eq:conditions-MD-gen} holds), which is equivalent to the
convergence of $W(r)$ to $\Lambda^\prime(0)$;
\item the weak convergence of $\sqrt{v_r}(W(r)-\mathbb{E}[W(r)])$ to the 
centered Normal distribution with variance $\Lambda^{\prime\prime}(0)$ (case 
$\gamma_r=1$; note that only the second condition in \eqref{eq:conditions-MD-gen}
holds).
\end{itemize}
\end{Claim}

We present an illustrative example. We consider the case where 
$$W(r):=\frac{X_1+\cdots+X_r}{r}$$
(here $r$ is an integer) and $\{X_n:n\geq 1\}$ are i.i.d. real valued random variables;
moreover we assume that $\mathbb{E}[e^{\theta X_1}]$ is finite in a neighborhood of 
the origin $\theta=0$, and therefore all the (common) moments of the random variables
$\{X_n:n\geq 1\}$ are finite. We denote the common mean by $\mu$ and the common variance
by $\sigma^2$. Then:
\begin{itemize}
\item as far as Claim \ref{claim:LDMD} is concerned, we have the initial LDP $v_r=r$ and
$\Lambda(\theta):=\log\mathbb{E}[e^{\theta X_1}]$ (in this case G\"{a}rtner Ellis Theorem
is not needed because we can refer to the Cram\'er Theorem, see e.g. Theorem 2.2.3 in 
\cite{DemboZeitouni}), and the class of LDPs (see e.g. Theorem 3.7.1. in 
\cite{DemboZeitouni});
\item as far as Claim \ref{claim:mean-variance-fill-the-gap} 
is concerned, the limits in \eqref{eq:mean-variance-formula} trivially hold with
$\mathbb{E}[W(r)]=\mu=\Lambda^\prime(0)$ and $v_r\mathrm{Var}[W(r)]=\sigma^2=\Lambda^{\prime\prime}(0)$
for all integer $r\geq 1$; moreover moderate deviations fill the gap between
the regimes of the classical Law of Large Numbers (for centered random variables) and 
Central Limit Theorem.
\end{itemize}

Finally, we concentrate the attention on the cases studied in Propositions 
\ref{prop:LD} and \ref{prop:MD} in this paper. We already remarked that the
random variables in Proposition \ref{prop:LD} are not centered, but they 
converge to zero as $r\to\infty$ because the rate function $J$ uniquely 
vanishes at zero. The Mac Laurin formula of order 2 of $J$ (for nonnegative
arguments) is $\tilde{J}$, and therefore we have some analogy with what happens
for $\Lambda^*$ and $\tilde{\Lambda}^*$ in Remark \ref{rem:Taylor-formula-order-2};
in fact $J$ is the restriction of $\tilde{J}$ on $[0,t]$. Furthermore, one could 
investigate if we have limits as the ones in \eqref{eq:mean-variance-formula}.
Here we take into account Theorem 2.2 in \cite{IafrateOrsingher} (see also Remark 
2.4 in \cite{IafrateOrsingher} for the mean value). Firstly, we have
$$\mathbb{E}[T_{\mu\sqrt{r},t}]=\frac{1-e^{-\frac{\mu^2r}{2}t}}{\mu^2r},$$
and therefore $\lim_{r\to\infty}\mathbb{E}[T_{\mu\sqrt{r},t}]=0$. Moreover
\begin{multline*}
\mathrm{Var}[T_{\mu\sqrt{r},t}]=\mathbb{E}[T_{\mu\sqrt{r},t}^2]-\mathbb{E}^2[T_{\mu\sqrt{r},t}]
=\frac{3}{4}\int_0^tae^{-\frac{\mu^2r}{2}a}da-\left(\frac{1-e^{-\frac{\mu^2r}{2}t}}{\mu^2r}\right)^2\\
=\frac{3}{4}\left(-\frac{2t}{\mu^2r}e^{-\frac{\mu^2r}{2}t}+\left(\frac{2}{\mu^2r}\right)^2(1-e^{-\frac{\mu^2r}{2}t})\right)
-\frac{(1-e^{-\frac{\mu^2r}{2}t})^2}{\mu^4r^2}\\
=-\frac{3t}{2\mu^2r}e^{-\frac{\mu^2r}{2}t}+\frac{3}{\mu^4r^2}(1-e^{-\frac{\mu^2r}{2}t})-\frac{(1-e^{-\frac{\mu^2r}{2}t})^2}{\mu^4r^2};
\end{multline*}
so, in order to have a finite and positive limit, we have to take a different
scaling, i.e.
\begin{equation}\label{eq:variance-limit}
\lim_{r\to\infty}r^2\mathrm{Var}[T_{\mu\sqrt{r},t}]=\frac{2}{\mu^4}
\end{equation}
(on the contrary the limit for the variance with the same scaling as in 
\eqref{eq:mean-variance-formula}, and therefore with the speed $v_r=r$ as in
Proposition \ref{prop:LD}, is equal to zero). We can also say that, as happens 
for Theorem 3.7.1 in \cite{DemboZeitouni}, Proposition \ref{prop:MD} fill the 
gap between two regimes (as $r\to\infty$). We have again a convergence to zero
for $T_{\mu\sqrt{r},t}$ (case $\gamma_r=1/r$) but, on the contrary, 
$rT_{\mu\sqrt{r},t}$ (case $\gamma_r=1$) converges weakly to the distribution 
of a nonnegative random variable $Y$ with distribution function $G$ defined by
$$G(a)=1-\frac{2}{\pi}\int_0^\infty\frac{e^{-\frac{\mu^2a}{2}(1+y^2)}}{1+y^2}dy,\ \mbox{for all}\ a\geq 0.$$
Finally we compute the variance of $Y$. We have
\begin{multline*}
\mathrm{Var}[Y]=\mathbb{E}[Y^2]-\mathbb{E}^2[Y]
=\int_0^\infty \underbrace{P(Y^2>a)}_{=P(Y>\sqrt{a})}da-\left(\int_0^\infty P(Y>a)da\right)^2\\
=\int_0^\infty\frac{2}{\pi}\int_0^\infty\frac{e^{-\frac{\mu^2\sqrt{a}}{2}(1+y^2)}}{1+y^2}dyda
-\left(\int_0^\infty\frac{2}{\pi}\int_0^\infty\frac{e^{-\frac{\mu^2a}{2}(1+y^2)}}{1+y^2}dyda\right)^2;
\end{multline*}
moreover, after some manipulations (in particular, for the second equality we consider the change of 
variable $\bar{a}=\sqrt{a}$ in the first term and, later, we put $a$ in place of $\bar{a}$), we obtain
\begin{multline*}
\mathrm{Var}[Y]=\frac{2}{\pi}\int_0^\infty\frac{1}{1+y^2}\int_0^\infty e^{-\frac{\mu^2\sqrt{a}}{2}(1+y^2)}dady
-\frac{4}{\pi^2}\left(\int_0^\infty\frac{1}{1+y^2}\int_0^\infty e^{-\frac{\mu^2a}{2}(1+y^2)}dady\right)^2\\
=\frac{2}{\pi}\int_0^\infty\frac{1}{1+y^2}\int_0^\infty e^{-\frac{\mu^2a}{2}(1+y^2)}2adady
-\frac{4}{\pi^2}\left(\int_0^\infty\frac{1}{1+y^2}\int_0^\infty e^{-\frac{\mu^2a}{2}(1+y^2)}dady\right)^2\\
=\frac{4}{\pi}\int_0^\infty\frac{1}{1+y^2}\frac{\Gamma(2)}{(\frac{\mu^2}{2}(1+y^2))^2}dy
-\frac{4}{\pi^2}\left(\int_0^\infty\frac{1}{\frac{\mu^2}{2}(1+y^2)^2}dy\right)^2\\
=\frac{16}{\pi\mu^4}\int_0^\infty\frac{1}{(1+y^2)^3}dy-\frac{16}{\pi^2\mu^4}\left(\int_0^\infty\frac{1}{(1+y^2)^2}dy\right)^2;
\end{multline*}
finally we consider the change of variable $y=\tan\theta$ in both integrals and we get
\begin{multline*}
\mathrm{Var}[Y]=\frac{16}{\pi\mu^4}\int_0^{\pi/2}\cos^4\theta d\theta-\frac{16}{\pi^2\mu^4}\left(\int_0^{\pi/2}\cos^2\theta d\theta\right)^2\\
=\frac{16}{\pi\mu^4}\underbrace{\left(\left[\frac{\cos^3\theta\sin\theta}{4}
+\frac{3}{4}\left(\frac{\theta+\sin\theta\cos\theta}{2}\right)\right]_{\theta=0}^{\theta=\pi/2}
-\frac{1}{\pi}\left(\left[\frac{\theta+\sin\theta\cos\theta}{2}\right]_{\theta=0}^{\theta=\pi/2}\right)^2\right)}_{=\frac{3\pi}{16}
-\frac{1}{\pi}\frac{\pi^2}{16}}=\frac{2}{\mu^4}.
\end{multline*}
Thus, in some sense, we have some analogy with the classical moderate deviation results, 
because the limit of rescaled variance in \eqref{eq:variance-limit} coincides with the 
variance of the weak limit $Y$ of $rT_{\mu\sqrt{r},t}$ (concerning the case $\gamma_r=1$).

\paragraph{Acknowledgements.} We thank an anonymous referee for
the careful reading of the first version of the manuscript. His/her
comments allowed us to correct two inaccuracies.


\begin{thebibliography}{spc}
\bibitem{AlbrecherAsmussen}
Albrecher, H., Asmussen, S. (2010). Ruin Probabilities. Second 
Edition. World Scientific, Singapore.
\bibitem{DemboZeitouni}
Dembo, A., Zeitouni, O. (1998). Large Deviations Techniques and
Applications. Second Edition. Springer, New York.
\bibitem{DuffyLewisSullivan}
Duffy, K., Lewis, J.T., Sullivan, W.G. (2003). Logarithmic
asymptotics for the supremum of a stochastic processes. Annals
of Applied Probability, 13, 430--445.
\bibitem{EmbrechtsKluppelbergMikosch}
Embrechts, P., Kl\"uppelberg, C., Mikosch, T. (1997). Modelling
Extremal Events. Springer, Berlin.
\bibitem{GaneshTorrisi}
Ganesh, A., Torrisi, G.L. (2008). Large Deviations of the 
Interference in a Wireless Communication Model. IEEE Transaction
on Information Theory, 54, 3505--3517.
\bibitem{IafrateOrsingher}
Iafrate, F., Orsingher, E. (2020). The last zero-crossing of an
iterated Brownian motion with drift. Stochastics, 92, 356--378.

\end{thebibliography}
\end{document}